\documentclass[11pt]{article}
\usepackage[margin=1in]{geometry}
\usepackage{amsmath,amssymb,amsthm}
\usepackage[T1]{fontenc}
\usepackage{lmodern}
\usepackage{microtype}
\usepackage[colorlinks=true,allcolors=blue]{hyperref}
\newtheorem{theorem}{Theorem}
\newtheorem{lemma}[theorem]{Lemma}

\theoremstyle{definition}
\newtheorem{definition}[theorem]{Definition}

\newcommand{\Z}{\mathbb Z}
\newcommand{\N}{\mathbb N_0}
\newcommand{\supp}{\operatorname{supp}}
\newcommand{\nuu}{\nu_2}
\title{Avoiding four-term progressions from finitely many starts}
\author{William Thompson\\
{\small Harvard University}\\
{\small \href{mailto:willthompsonbusiness@gmail.com}{\texttt{willthompsonbusiness@gmail.com}}}}
\date{September 10, 2026}
\hypersetup{
  pdftitle={Avoiding four-term progressions from finitely many starts},
  pdfauthor={William Thompson},
  pdfsubject={Arithmetic progressions in enumerations of the integers}
}
\begin{document}
\maketitle
\begin{abstract}
For every finite set $A\subset\Z$, we construct a bijection $p:\N\to\Z$
beginning with $0,1$ that contains no four-term arithmetic progression,
in occurrence order, whose first value belongs to $A$. More generally,
for each $n\ge2$ the enumeration can begin with
$0,1,3,\ldots,2^{n-1}-1$ and simultaneously avoid every such progression
starting in $A$ or among these first $n$ entries. The construction uses
finitely branching prerequisite relations and an explicit bounded integer
potential. This proves finite prerequisite closure and gives an exhaustive
enumeration, rather than merely a total order. We also give an extension
theorem for finite prefixes with compatible binary-tail constraints.
These results do not determine whether every enumeration of the integers
contains an ordered four-term arithmetic progression.
\end{abstract}

\section{Statement and context}

An \emph{enumeration} of $\Z$ means a bijection $p:\N\to\Z$.
Write $x\prec_p y$ when $p^{-1}(x)<p^{-1}(y)$.
An ordered $\ell$-term arithmetic progression is a sequence
\[
 a,a+d,\ldots,a+(\ell-1)d,
 \qquad d\in\Z\setminus\{0\},
\]
whose terms occur in the displayed order. Both signs of $d$ are included,
and the positions need not be consecutive or themselves form an
arithmetic progression. We call $a$ its \emph{starting value}.

Erd\H{o}s problem~195 asks for the greatest length forced in every
enumeration of $\Z$ \cite{Bloom}. Geneson constructed an enumeration
avoiding ordered six-term progressions \cite{Geneson19}, and Adenwalla
improved this to five-term avoidance \cite{Adenwalla}.
Every enumeration contains an ordered three-term progression: if $a$ is
its first entry and $b$ is its first entry greater than $a$, then
$2b-a>b$ occurs after $b$.
Thus the greatest unavoidable length is either three or four.
Geneson's more recent density result produces four-permutable subsets
with lower symmetric densities arbitrarily close to one \cite{Geneson26};
it does not give a four-free enumeration of all integers.

Finite-closure methods also appear in K\'arolyi and Komj\'ath's construction
of an order of $\mathbb Q$ of type $\omega$ with no ordered six-term
arithmetic progression \cite[Lemma~8]{Karolyi}. Here we keep all integers and restrict the possible starting values.
The following statement also permits a specified family of arbitrarily
long initial segments.

\begin{theorem}\label{thm:main}
Let $A\subset\Z$ be finite and let $n\ge2$.
There is an enumeration $p$ of $\Z$ such that
\[
 p(j)=2^j-1\quad(0\le j<n),
\]
and no ordered four-term arithmetic progression has its starting value
in $A\cup\{2^j-1:0\le j<n\}$.
\end{theorem}

Taking $n=2$ gives avoidance from any prescribed finite set while keeping
the first entries $0,1$. Taking $A=\varnothing$ shows that for every $n$
there is an enumeration in which no four-term progression starts among
the first $n$ entries. The enumeration depends on $A$ and $n$.
Theorem~\ref{thm:main} provides no compatible limiting construction as
these parameters increase.

The length four is optimal in the first assertion: every enumeration
beginning with $0$ contains a three-term progression starting at $0$,
by the first-greater-entry argument above.

\section{Binary orders and finite prefixes}

For $u\ne0$, let $\nuu(u)$ denote the largest $t\ge0$ for which $2^t$ divides
$u$. This definition applies to negative integers as well.
The residue class $r\pmod{2^t}$ has two children modulo $2^{t+1}$.
Choose, independently at each such node, which child precedes the other.
Distinct integers $x,y$ are compared at their first differing binary
digit, of depth $\nuu(x-y)$, using the chosen order of its two children.
We call the resulting order a \emph{binary order}.

This is a total order: any finite set of distinct integers has distinct
residues modulo some $2^K$, and its comparisons are those of the finite
ordered binary tree of depth $K$. These finite orders agree as $K$
increases. No assertion about order type is implicit in this definition.
Parity-recursive orders are standard in arithmetic-progression avoidance;
see Davis et al.\ \cite{Davis} and the explicit integer ordering of
Ardal, Brown and Jungi\'c \cite{Ardal}.

\begin{lemma}\label{lem:binary}
A binary order contains no ordered three-term arithmetic progression.
\end{lemma}
\begin{proof}
For $a,a+d,a+2d$, put $t=\nuu(d)$.
The three values agree modulo $2^t$. At depth $t$ the endpoints lie in
one child and the middle value in the other. The middle value therefore
precedes both endpoints or follows both. It cannot occur between them.
\end{proof}

Let $P=(p_0,\ldots,p_{s-1})$ be a nonempty finite word of distinct integers,
and put $S=\supp(P)$. An inequality between elements of $S$ refers to their
order in $P$. Values in $S$ will be called \emph{old}; other values are new.

\begin{definition}\label{def:certificate}
The word $P$ has a \emph{binary-tail certificate} if the following hold.
\begin{enumerate}
\item Whenever $a\prec_P b\prec_P c$ is a three-term arithmetic
progression in $P$, its continuation $2c-b$ belongs to $S$ and precedes $c$.
\item There is a binary order $\triangleleft$ such that, for each old
pair $a\prec_P b$, if
\[
 c=2b-a\notin S,\qquad e=3b-2a\notin S,
\]
then $e\triangleleft c$.
\end{enumerate}
\end{definition}

This is a finite, decidable condition. Each inequality in the second
part prescribes one child preference at the first differing digit of
$c,e$; the condition holds exactly when no node receives opposite
prescriptions. The first part also implies that $P$ itself is four-free.
Indeed, the fourth term of an ordered four-term progression in $P$ would
have to precede its third term.

One motivation for this certificate is that $P$ followed by its new values
in the binary order $\triangleleft$ is a four-free total order. A progression
with at most one old term would leave a three-term progression in the tail;
two old terms violate the second condition; and at least three old terms
violate the first. This total order need not be an enumeration.

\section{An exhaustive extension theorem}

\begin{theorem}\label{thm:extension}
Suppose $P$ has a binary-tail certificate, and let $A\subset\Z$ be finite.
There is an enumeration of $\Z$ beginning with $P$ which avoids every
ordered four-term arithmetic progression whose starting value lies in
$A\cup\supp(P)$.
\end{theorem}

\begin{proof}
Set $B=A\cup S$. Choose $K\ge1$ with
\[
 m=2^K>\max B-\min B.
\]
Distinct elements of $B$ then have distinct residues modulo $m$.
Every prescribed old-pair comparison is decided below depth $K$, since
$e-c=b-a$ and $0<|b-a|<m$.
Complete these child preferences to an order $\theta$ of all residues
modulo $m$. Write $r(x)\in\{0,\ldots,m-1\}$ for the rank of the residue
of $x$ in this order, starting at zero.

\smallskip\noindent\textbf{Prerequisites.}
For each $a\in B$ and each $t\ge0$, choose one of two symbols $F_a(t)$
and $G_a(t)$. For $t<K$, choose $F_a(t)$ exactly when the child containing
$a$ is preferred at its depth-$t$ node in $\theta$; otherwise choose $G_a(t)$.
For $t\ge K$, choose $F_a(t)$ when $t-K$ is even, and $G_a(t)$ when it is odd.

For every target $x\notin S$ and every $a\in B$ with $x\ne a$, impose
the following prerequisites when their stated conditions hold:
\begin{align}
 2x-a&\longrightarrow x
 &&\text{if }F_a(\nuu(x-a))\text{ is chosen};\label{eq:F}\\
 \frac{3x-a}{2}&\longrightarrow x
 &&\text{if }2\mid(x-a)\text{ and }G_a(\nuu(x-a)-1)\text{ is chosen}.
 \label{eq:G}
\end{align}
An arrow $y\longrightarrow x$ means that $y$ must be listed before $x$.
Also impose $e\longrightarrow c$ for each old-pair comparison in
Definition~\ref{def:certificate}. There are no prerequisites targeting $S$.
Each target has finitely many immediate prerequisites, at most
$D=2|B|+\binom{s}{2}$.

\smallskip\noindent\textbf{A decreasing potential.}
First consider an arrow $y\longrightarrow x$ of type~\eqref{eq:F},
and put $v=\nuu(x-a)$. If $v<K$, then $y$ and $a$ are in the same
depth-$v$ child, while $x$ is in the other child. Choice $F$ makes
the former child preferred, so $r(y)<r(x)$. If $v\ge K$, then
$r(y)=r(x)$ and
\[
 \nuu(y-a)=v+1.
\]
For an arrow of type~\eqref{eq:G}, let $v=\nuu(x-a)\ge1$.
At depth $v-1$, the values $x$ and $a$ are in the same child and $y$
is in the other. If $v-1<K$, choice $G$ makes the other child preferred,
so again $r(y)<r(x)$. Otherwise $r(y)=r(x)$ and
\[
 \nuu(y-a)=v-1,
\]
because $y-a=3(x-a)/2$.
Finally, every old-pair arrow strictly decreases rank by the choice of
$\theta$ and the distinct residues of its endpoints.

Define $\epsilon(x)=1$ precisely when there is an $a\in B$ with
$x\equiv a\pmod m$, $x\ne a$, and $\nuu(x-a)-K$ even;
otherwise set $\epsilon(x)=0$. Such an anchor $a$ is unique when it exists.
Set
\[
 V(x)=2r(x)+\epsilon(x)\in\{0,\ldots,2m-1\}.
\]
Every arrow $y\longrightarrow x$ satisfies $V(y)<V(x)$.
This follows immediately if $r(y)<r(x)$.
For an equal-rank arrow of type~\eqref{eq:F}, the relative valuation
$v-K$ is even at $x$ and becomes odd at $y$.
For an equal-rank arrow of type~\eqref{eq:G}, choice $G$ means
$v-1-K$ is odd, so $v-K$ is even at $x$ and $v-1-K$ is odd at $y$.
Thus in both cases $\epsilon$ decreases from one to zero. The anchor
used in this calculation is the unique anchor in that residue class.

It follows that a path followed backwards from target to prerequisite
has at most $2m-1$ edges. Since each vertex has at most $D$ immediate
prerequisites, the prerequisite closure of any vertex is finite
(of size at most $\sum_{j=0}^{2m-1}D^j$). The graph is also acyclic.
This argument includes the boundary $v=K$ in~\eqref{eq:G}: its comparison
is decided at depth $K-1$ and strictly decreases rank.

\smallskip\noindent\textbf{Enumeration.}
Start with $P$ and request targets successively in the order
$0,1,-1,2,-2,\ldots$. For each missing target, append its still missing
prerequisite closure, including that target, in a topological order.
This is possible because the closure is finite and acyclic. At every
stage the listed set is closed under prerequisites: this is initially
true because old values have no prerequisites, and appending a closure
preserves it. Thus every arrow is respected. Every stage is finite,
no value is repeated, and every integer is eventually requested.
The result is an enumeration of all $\Z$ extending $P$.

\smallskip\noindent\textbf{Avoidance.}
Suppose, for a contradiction, that $a,b,c,e$ is an ordered four-term
progression with $a\in B$. If $b\notin S$, then $c,e\notin S$ as well,
since all old values occur in the initial prefix. Put $t=\nuu(b-a)$.
Choice $F_a(t)$ supplies $c\longrightarrow b$ via~\eqref{eq:F}.
Choice $G_a(t)$ supplies $e\longrightarrow c$ via~\eqref{eq:G}.
Either contradicts the progression's occurrence order.

If $b\in S$, then also $a\in S$, since $a$ occurs before $b$.
If $c\in S$, the first certificate condition puts $e$ before $c$.
If $c\notin S$ but $e\in S$, the same reversal holds because $e$ is old.
If $c,e\notin S$, their old-pair prerequisite puts $e$ before $c$.
These cases exhaust all possibilities. The valuation identities and
child comparisons used throughout hold for either sign of the difference.
\end{proof}

\section{Prescribed initial positions}

\begin{proof}[Proof of Theorem~\ref{thm:main}]
Let $P_n=(2^j-1)_{j=0}^{n-1}$.
We verify its binary-tail certificate and apply
Theorem~\ref{thm:extension} with the given set $A$.

For an old pair $a=2^i-1$, $b=2^j-1$ with $i<j$, its third continuation
$c=2b-a$ satisfies
\[
 c+1=2^i(2^{j+1-i}-1).
\]
The factor in parentheses is odd and greater than one, so $c+1$ is not a
power of two. Thus $P_n$ contains no ordered three-term progression,
and the first certificate condition is vacuous.

For the fourth continuation $e=3b-2a$, direct reduction gives
\[
 c\equiv2^i-1\pmod{2^{i+1}},\qquad
 e\equiv2^{i+1}-1\pmod{2^{i+1}}.
\]
Consequently $c,e$ first differ at depth $i$, at the node
$2^i-1\pmod{2^i}$, and their bit-$i$ values are respectively zero and one.
Every unresolved comparison $e\triangleleft c$ therefore requires the
one-child first at that node. All requirements agree, so the second
certificate condition holds as well.
\end{proof}

The finite bounds in the prerequisite construction do not supply a
four-free enumeration when all starting values are protected at once.
The modulus, residue order, and resulting positions can change when
the finite anchor set changes. In particular, neither a limiting
enumeration nor bounds on individual positions that are uniform as
$A$ and $n$ increase follow from Theorem~\ref{thm:main}. The question of four-term avoidance
on all of $\Z$ remains open.

\newpage
\section*{Computational supplement and use of AI}

The supplementary Python program implements the prerequisite construction
with exact integer arithmetic and independently checks finite instances
of the potential inequalities, prefix conditions, and rooted avoidance.
These computations are consistency checks; the proofs of termination,
coverage, and avoidance are the symbolic arguments above.

The author directed the research and developed the results presented here
with assistance from GPT-6 Astra. The author designed structured prompts
to guide the search for constructions, test proposed mechanisms, and
challenge the arguments, and personally reviewed and verified every proof
in this manuscript.

Astra assisted with the development of the construction and proofs,
verification code, literature research, and manuscript preparation.
Additional Astra runs were used to reconstruct and challenge the arguments
and to produce a separate verification implementation. These computational
and AI checks supplemented the author's mathematical review; they do not
constitute external peer review or formal proof-assistant verification.
The author takes responsibility for all statements, references, and
computations reported here.


\begin{thebibliography}{99}
\bibitem{Adenwalla}
S.~Adenwalla, \emph{Avoiding monotone arithmetic progressions in permutations
of integers}, Discrete Mathematics \textbf{347} (2024), no.~11, 114183.
\href{https://doi.org/10.1016/j.disc.2024.114183}{doi:10.1016/j.disc.2024.114183}.
\bibitem{Ardal}
H.~Ardal, T.~Brown and V.~Jungi\'c, \emph{Chaotic orderings of the rationals
and reals}, American Mathematical Monthly \textbf{118} (2011), no.~10,
921--925. \href{https://doi.org/10.4169/amer.math.monthly.118.10.921}{doi:10.4169/amer.math.monthly.118.10.921}.
\bibitem{Bloom}
T.~F.~Bloom, \emph{Erd\H{o}s Problem \#195},
\url{https://www.erdosproblems.com/195}.
\bibitem{Davis}
J.~A.~Davis, R.~C.~Entringer, R.~L.~Graham and G.~J.~Simmons,
\emph{On permutations containing no long arithmetic progressions},
Acta Arithmetica \textbf{34} (1977), no.~1, 81--90.
\href{https://doi.org/10.4064/aa-34-1-81-90}{doi:10.4064/aa-34-1-81-90}.
\bibitem{Geneson19}
J.~Geneson, \emph{Forbidden arithmetic progressions in permutations of
subsets of the integers}, Discrete Mathematics \textbf{342} (2019), no.~5,
1489--1491. \href{https://doi.org/10.1016/j.disc.2019.02.004}{doi:10.1016/j.disc.2019.02.004}.
\bibitem{Geneson26}
J.~Geneson, \emph{Density bounds for permutations avoiding monotone
arithmetic progressions},
\href{https://arxiv.org/abs/2608.12604}{arXiv:2608.12604} (2026).
\bibitem{Karolyi}
G.~K\'arolyi and P.~Komj\'ath, \emph{Well ordering groups with no monotone
arithmetic progressions}, Order \textbf{34} (2017), 299--306.
\href{https://doi.org/10.1007/s11083-016-9400-5}{doi:10.1007/s11083-016-9400-5}.
\end{thebibliography}
\end{document}